\documentclass[reqno]{amsart}
\usepackage{amssymb}
\usepackage{hyperref}

\title[ ]
{Quasi-Fredholm, Saphar Spectra For Integrated semigroup  }

\subjclass[2000]{47D03, 47A10}
\keywords{Integrated semigroup, generator, Saphar, quasi-Fredholom }

\author[A. TAJMOUATI,  H. BOUA,  M.KARMOUNI  ]
{ A. TAJMOUATI,  H. BOUA,  M.KARMOUNI    }

\address{A. TAJMOUATI,  H. BOUA\newline
 Sidi Mohamed Ben Abdellah University,
 Faculty of Sciences Dhar Al Mahraz, Laboratory of Mathematical Analysis and Applications, Fez, Morocco.}
 \email{abdelaziz.tajmouati@usmba.ac.ma}
\email{hamid.boua@usmba.ac.ma}
\address{M. KARMOUNI \newline
Cadi Ayyad University, Multidisciplinary Faculty, Safi, Morocco.}
\email{med89karmouni@gmail.com}

\subjclass[2010]{47B47, 47B20, 47B10}
\keywords{Integrated semigroup, quasi-Fredholm operator, Saphar spectrum.}
\newtheorem{theorem}{Theorem}[section]
\newtheorem{definition}{Definition}[section]
\newtheorem{remark}{Remark}

\newtheorem{lemma}{Lemma}[section]
\newtheorem{proposition}{Proposition}[section]
\newtheorem{corollary}{Corollary}[section]

\begin{document}
\maketitle

\begin{abstract}
In this paper, we show a spectral inclusion of  integrated semigroups  for Saphar, essentially Saphar and  quasi-Fredholm spectra.
\end{abstract}

\section{Introduction And Preliminaries}
Throughout, $X$ denotes a complex Banach space, let $A$ be a closed linear operator on $X$ with domain  $D(A)$, we denote by $A^*$,  $R(A)$, $N(A)$, $ R^{\infty}(A)=\bigcap_{n\geq0}R(A^n)$, $\sigma(A)$,
 respectively the adjoint, the range, the null space, the hyper-range, the spectrum of $A$.\\
Recall that a  closed operator $A$ is said to be Kato operator or
semi-regular if $R(A)$ is closed
and $N(A)\subseteq R^{\infty}(A)$. Denote by $\rho_{K}(A)$ :
$\rho_{K}(A)=\{\lambda\in\mathbb{C}: A-\lambda I\mbox{  is Kato }
\}$ the Kato resolvent  and
$\sigma_{K}(A)=\mathbb{C}\backslash\rho_{K}(A)$ the Kato spectrum
of $A$. It is well known that $\rho_{K}(A)$ is an open subset of
$\mathbb{C}$.\\
 For subspaces $M$, $N$ of $X$ we write $M\subseteq^{e}N$ ($M$ is essentially contained in $N$) if there exists a finite-dimensional subspace $F\subset X$ such that $M\subseteq N+F$.\\
 A closed operator $S$ is called a generalized inverse of $A$ if $R(A)\subseteq D(S)$,  $R(S)\subseteq D(A)$, $ASA=A$ on $D(A)$  and $SAS=S$ on $D(S)$, wich equivalent to the fact that $R(A)\subseteq D(S)$,  $R(S)\subseteq D(A)$, $ASA=A$ on $D(A)$. \\
 A closed operator  $A$ is called a  Saphar operator if $A$ has a generalized
 inverse and $ N(A)\subseteq R^{\infty}(A)$, which equivalent to the fact that $A$ is
 Kato operator  and has a generalized inverse, see \cite{Mul}.\\
 If we assume in the definition above that $ N(A)\subseteq^{e} R^{\infty}(A)$, $A$ is said to be a essentially Saphar operator.
 The  Saphar and essentially Saphar spectra are defined by $$\sigma_{Sap}(A)=\{\lambda\in \mathbb{C}:\>\> A-\lambda \mbox{ is  not  Saphar} \}.$$
 $$\sigma^{e}_{Sap}(A)=\{\lambda\in \mathbb{C}:\>\> A-\lambda \mbox{ is  not  essentially Saphar} \}.$$
 $\sigma_{Sap}(A)$ is a compact non empty set of $\mathbb{ C}$ and we have $$\partial \sigma(A)\subseteq\sigma_{K}(A)\subseteq\sigma_{Sap}(A)\subseteq \sigma(A)$$
  Integrated semigroups were first defined by Arendt \cite{A} for integer-valued
$\alpha$. Arendt showed that certain natural classes of operators, such as adjoint semigroups
of $C_0$ semigroups on non-reflexive Banach spaces, give rise to integrated
semigroups which are not integrals of $C_0$  semigroups.\\

 A family of bounded linear operators $(S(t))_{t\geq 0}$, on a Banach space X is called an integrated semigroup iff \\
 \begin{enumerate}
   \item  $S(0) = 0$.
   \item  $S(t)$ is strongly continuous in $t\geq 0$.
   \item $S(r)S(t) = \int_0^r(S(\tau + t) - S(\tau))d\tau = S(t) S(r)$.
 \end{enumerate}

The differentation spaces $C^n$, $n \geq 0$, are defined by $C^0 = X$ and
\begin{center}
$C^n = \{x \in  X : S(.)x \in C^n(\mathbb{R}^+; X)\}$
\end{center}
Using this notion,  3) can equivalently be formulated by
$$S(t)x \in C^1 \mbox{ and }  S'(r)S(t) = S(r + t) - S(r)$$
The set $N= \{x\in X; S(t)x=0, \forall t\geq 0\}$ is called the degeneration space of the integrated semigroup $(S(t))_{t\geq 0}$.
 $(S(t))_{t\geq 0}$ is called non-degenerate if $N = \{0\}$ and degenerate otherwise.

The generator $A : D(A) \subseteq X \rightarrow X$ of a non-degenerate integrated semigroup $(S(t))_{t\geq 0}$ is defined as follows: $x \in D(A)$ and $Ax = y$ iff $x \in C^1$ and $S'(t)x-x=S(t)y$ for $t \geq 0 $.

\begin{center}
 $C^2 \subseteq  D(A) \subseteq C^1$ and $Ax=S''(0)x$ for $X \in C^2$. Moreover $AC^2 \subseteq C^1$.
\end{center}
Note that $A$ is a closed linear operator.  \\
$S(t) : C^1 \rightarrow C^2 \subseteq D(A)$ and $AS(t) x = S''(0) S(t) x =S'(t) x - x$. Further $AS(t) x = S(t) Ax$ for all $x \in D(A)$.\\
$\int_0^t S(r)dr$ maps $X$ into $D(A)$ and $A\int_0^t S(r)xdr =S(t)x -tx$.\\
A non-degenerate integrated semigroup is uniquely determined by its generator. \\
Let $u: [0, T)\rightarrow X$ be continuous such that $\int_0^t u(s)ds \in D(A)$   and  $A(\int_0^t u(s)ds) =u(t)$, for $0 \leq t \leq T$. Then $u = 0$ in $[0, T)$. Arendt \cite{A} showed that if $A$
generates $S_t$ as an $n$-times integrated semigroup, then the Abstract Cauchy
Problem
$u'(t) = Au(t), u(0) = x$ has a classical solution for all $x \in D(A^{n+1})$.\\
The aim of this  paper is to established the relationship between the integrated semigroup and its generator, more precisely, we show that
$$\int_0^t e^{t\sigma_{*}(A)}ds \subseteq \sigma_{*}(S(t))$$
where $\sigma_{*}$ run the Saphar, essential Saphar and quasi-Fredholm spectra.
\section{Spectral Inclusions  For Saphar Spectrum}
We start by some lemmas, they  will be needed in the sequel.
\begin{lemma}\label{l1s}
Let $A$ be the generator of a non-degenerate integrated semigroup $(S(t))_{t\geq 0}$, $D_\lambda(t)x=\int_0^t e^{\lambda (t-s)}S(s)xds$. Then, for all $\lambda \in \mathbb{C}$, $t\geq 0$, and $n\in \mathbb{N}$,
\begin{enumerate}
\item $(\int_0^t e^{\lambda s}ds-S(t))x =(\lambda -A)D_\lambda(t)x,   \forall x\in X .$
\item $(\int_0^te^{\lambda s}ds-S(t))x =D_\lambda(t)(\lambda -A) x,  \forall x\in D(A).$
\item $R(\int_0^te^{\lambda s}ds-S(t))^n\subseteq R(\lambda -A)^n$.
\item $N((\lambda -A)^n)\subseteq N(\int_0^te^{\lambda s}ds-S(t))^n$.
\end{enumerate}
\end{lemma}

\begin{proof}
\begin{enumerate}
\item for all $r$, $t \in [0, +\infty[$  and  $x\in X$ we have :
\begin{eqnarray*}
  S(r)D_\lambda(t)x &=& S(r)\int_0^t e^{\lambda (t-s)}S(s)xds \\
                    &=&\int_0^t e^{\lambda (t-s)}S(r)S(s)xds  \\
                    &=&\int_0^t \int_0^re^{\lambda (t-s)}[S(\tau +s)-S(\tau)]xd\tau ds  \\
                    &=&\int_0^r \int_0^te^{\lambda (t-s)}[S(\tau +s)-S(\tau)]xds d\tau
\end{eqnarray*}
Then, for all  $x\in X$,  $D_\lambda(t)x \in C^1$, and we have
\begin{eqnarray*}
  \frac{d}{dr}S(r)D_\lambda(t)x&=&\int_0^te^{\lambda (t-s)}[S(r +s)-S(r)]xds \\
                               &=&\int_0^te^{\lambda (t-s)}[S(r +s)-S(s)]xds + \int_0^te^{\lambda (t-s)}S(s)xds -\int_0^te^{\lambda (t-s)}S(r)xds\\
                               &=&\int_0^te^{\lambda (t-s)}\frac{d}{ds}[S(s)S(r)]xds -S(r)\int_0^te^{\lambda s}xds +D_\lambda(t)x\\
                               &=&S(r)S(t)x+\lambda S(r)D_\lambda(t)x -S(r)\int_0^te^{\lambda s}xds +D_\lambda(t)x\\
                               &=&S(r)[S(t)+\lambda D_\lambda(t)x -\int_0^te^{\lambda s}xds] +D_\lambda(t)x
\end{eqnarray*}
Therefore $D_\lambda(t)x \in D(A)$ and  $A D_\lambda(t)x=S(t)+\lambda D_\lambda(t)x -\int_0^te^{\lambda s}xds $. \\
Thus,
\begin{center}
$(\int_0^t e^{\lambda s}ds-S(t))x =(\lambda -A)D_\lambda(t)x$
\end{center}
\item  Let $x\in D(A)$, then
\begin{eqnarray*}
  D_\lambda(t)Ax&=&\int_0^t e^{\lambda (t-s)}S(s)Axds \\
                           &=&\int_0^t e^{\lambda (t-s)}(S'(s)x-x)ds \\
                           &=&\int_0^t e^{\lambda (t-s)}S'(s)xds -\int_0^t e^{\lambda s}xds  \\
                           &=&S(s)+\lambda D_\lambda(t)x -\int_0^t e^{\lambda s}xds
\end{eqnarray*}

Hence
\begin{center}
$(\int_0^te^{\lambda s}ds-S(t))x =D_\lambda(t)(\lambda -A) x$
\end{center}
\end{enumerate}

\end{proof}


\begin{lemma}\label{l1}
Let $(S(t))_{t\geq 0}$ be a non-degenerate integrated semigroup on $X$ with generator $A$. For  $\lambda \in \mathbb{C}$ and $t\geq 0$, let $L_{\lambda}(t)x=\int_0^t e^{-\lambda s}D_\lambda(s)x ds$, then:
\begin{enumerate}
\item $L_{\lambda}(t)$ is a bounded linear operator on $X$ .
\item  $\forall x\in X$, $L_{\lambda}(t)x\in D(A)$ and  $(\lambda -A)L_{\lambda}(t) + G_{\lambda}(t) D_\lambda(t)=\phi_{\lambda}(t)I$ with $G_{\lambda}(t)=e^{-\lambda t}I$ and $\phi_{\lambda}(t)=\int_0^t \int_0^{\tau} e^{-\lambda \sigma} x d\sigma d\tau$.
\item The operators $L_{\lambda}(t)$, $G_{\lambda}(t)$, $D_\lambda(t)$ and $(\lambda -A)$ are pairwise commute.
\end{enumerate}
\end{lemma}

\begin{proof}
\begin{enumerate}
\item  obvious.
\item For all $ r \geq 0$ we have :
\begin{eqnarray*}
     S(r)L_{\lambda}(t)x &=&\int_0^t e^{-\lambda \tau}S(r)D_\lambda(\tau)x d\tau\\
                         &=&\int_0^t \int_0^{\tau} e^{-\lambda \sigma} \int_0^r [S(u+\sigma)-S(u)]x du d\sigma d\tau \\
                         &=&\int_0^t \int_0^{\tau} \int_0^r e^{-\lambda \sigma} [S(u+\sigma)-S(u)]x du d\sigma d\tau \\
                         &=&\int_0^r \int_0^t \int_0^{\tau} e^{-\lambda \sigma} [S(u+\sigma)-S(u)]x d\sigma d\tau du
\end{eqnarray*}
Therefore, for all $x\in X$, $L_{\lambda}(t)x \in C^1$ and
\begin{eqnarray*}
 \frac{d}{dr}S(r)L_{\lambda}(t)x &=&\int_0^t \int_0^{\tau} e^{-\lambda \sigma} [S(r+\sigma)-S(r)]x d\sigma d\tau \\
                                 &=&\int_0^t \int_0^{\tau} e^{-\lambda \sigma} [S(r+\sigma)-S(\sigma)]x d\sigma d\tau + L_{\lambda}(t)x -S(r)\phi_{\lambda}(t)\\
                                 &=&\int_0^t \int_0^{\tau} e^{-\lambda \sigma} \frac{d}{d\sigma}S(\sigma)S(r)x d\sigma d\tau + L_{\lambda}(t)x -S(r)\phi_{\lambda}(t) \\
                                 &=&S(r)[e^{-\lambda t}D_{\lambda}(t)x + \lambda L_{\lambda}(t)x -\phi_{\lambda}(t)x]+L_{\lambda}(t)x
\end{eqnarray*}
Therefore, $AL_{\lambda}(t)x=e^{-\lambda t}D_{\lambda}(t)x + \lambda L_{\lambda}(t)x -\phi_{\lambda}(t)x $.\\
$(\lambda -A)L_{\lambda}(t) + G_{\lambda}(t) D_\lambda(t)=\phi_{\lambda}(t)I$ with $G_{\lambda}(t)=e^{-\lambda t}I.$
\item For all  $t \geq 0$,  $L_\lambda(t)$ and  $D_\lambda(t)$ commuting.\newline
Indeed, for  $t,s \geq 0$ we have  :
\begin{eqnarray*}
   D_{\lambda}(t)D_{\lambda}(s)x &=&\int_0^t e^{\lambda (t-u)}S(u)D_{\lambda}(s) x du  \\
                                 &=&\int_0^t e^{\lambda (t-u)}S(u)\int_0^s e^{\lambda (s-v)}S(v) x dv du \\
                                 &=& \int_0^t \int_0^s e^{\lambda (t-u)} e^{\lambda (s-v)}S(u)S(v) x dv du \\
                                 &=& \int_0^s e^{\lambda (s-v)}S(v) \int_0^t e^{\lambda (t-u)}S(u) x du dv \\
                                 &=&  D_{\lambda}(s)D_{\lambda}(t)x
\end{eqnarray*}
Therefore:
\begin{eqnarray*}
   L_{\lambda}(t)D_{\lambda}(t)x &=& \int_0^t e^{-\lambda u}D_{\lambda}(u)D_{\lambda}(t) x du  \\
                                 &=& \int_0^t e^{-\lambda u}D_{\lambda}(t)D_{\lambda}(u) x du  \\
                                 &=& D_{\lambda}(t) \int_0^t e^{-\lambda u}D_{\lambda}(u) x du \\
                                 &=& D_{\lambda}(t)L_{\lambda}(t)x
\end{eqnarray*}
For all $x\in D(A)$ we have:
\begin{eqnarray*}
   L_\lambda(t)(\lambda - A)x&=& \int_0^t e^{-\lambda s}D_\lambda(s)(\lambda - A)x ds\\
                             &=& \int_0^t e^{-\lambda s}(e^{\lambda s}-S(s))x ds \\
                             &=& \phi_{\lambda}(t)x-\int_0^t e^{-\lambda s}S(s)x ds \\
                             &=& \phi_{\lambda}(t)x - G_{\lambda}(t)D_\lambda(t)x \\
                             &=& (\lambda - A) L_\lambda(t)x
\end{eqnarray*}
For all  $x\in D(A)$   $(\lambda -A)G_{\lambda}(t)x= G_{\lambda}(t)(\lambda -A)x$  trivial\\
For all  $x\in D(A)$   $(\lambda -A)D_\lambda(t)x=D_\lambda(t)(\lambda -A)x$  see lemma \ref{l1s}
\end{enumerate}
\end{proof}


 \begin{lemma}\label{2}
Let $(S(t))_{t\geq 0}$ be a non-degenerate integrated semigroup with generator $A$. Then for all  $t>0$ we have:
\begin{center}
 $\int_0^t e^{\lambda t} ds -S(t)$ has a generalized inverse $\Longrightarrow$  $\lambda- A$ has a generalized inverse.
\end{center}
\end{lemma}
\begin{proof}
Suppose that  $\int_0^te^{\lambda s}ds-S(t)$ has a generalized inverse then there exists a $R\in \mathcal{B}(X)$ such that :
$$(\int_0^te^{\lambda s}ds-S(t))R(\int_0^te^{\lambda s}ds-S(t))=\int_0^te^{\lambda s}ds-S(t)$$

According to  lemma \ref{l1}, we have  $(\lambda -A)F_{\lambda}(t) + G_{\lambda}(t) D_\lambda(t)=\phi_{\lambda}(t)I$, then:

\begin{eqnarray*}
  \phi_{\lambda}(t) (\lambda- A) &= & (\lambda - A)F_{\lambda}(t)(\lambda-A)+ D_{\lambda}(t)G_{\lambda}(t)(\lambda-A) \\
                &= & (\lambda - A)F_{\lambda}(t)(\lambda-A)+ (\lambda-A)D_{\lambda}(t)G_{\lambda}(t)\\
                &= & (\lambda - A)F_{\lambda}(t)(\lambda-A)+ (\int_0^te^{\lambda s}ds-S(t))G_{\lambda}(t)  \\
                &= & (\lambda - A)F_{\lambda}(t)(\lambda-A)+ (\int_0^te^{\lambda s}ds-S(t))R(\int_0^te^{\lambda s}ds-S(t))G_{\lambda}(t)\\
                &= & (\lambda - A)F_{\lambda}(t)(\lambda-A)+ (\lambda-A)D_{\lambda}(t)R(\lambda-A)D_{\lambda}(t)G_{\lambda}(t)\\
                &= & (\lambda - A)F_{\lambda}(t)(\lambda-A)+ (\lambda-A)D_{\lambda}(t)R D_{\lambda}(t)G_{\lambda}(t)(\lambda-A)\\
                &= &  (\lambda - A)[F_{\lambda}(t)+ D_{\lambda}(t)R D_{\lambda}(t)G_{\lambda}(t)](\lambda-A)
\end{eqnarray*}

hence $\lambda-A$ has a generalized inverse.
\end{proof}

\begin{theorem}
Let $(S(t))_{t\geq 0}$ be a non-degenerate integrated semigroup with generator $A$. Then for all $t> 0$:
\begin{center}

$\int_0^t e^{t\sigma_{Sap}(A)}ds \subseteq \sigma_{Sap}(S(t)),\>\>\>\>\>\>
   \int_0^t e^{t\sigma^{e}_{Sap}(A)}ds \subseteq \sigma^{e}_{Sap}(S(t))$
\end{center}
\end{theorem}
\begin{proof}
Assume that $\int_0^te^{\lambda s}ds-S(t)$ is a Saphar operator, then $\int_0^te^{\lambda s}ds-S(t)$ has a generalized inverse and $N(\int_0^te^{\lambda s}ds-S(t))\subseteq R(\int_0^te^{\lambda s}ds-S(t))$. By  Lemma \ref{2},  $\lambda- A$ has a generalized inverse,
and  we have : $$N(\lambda -A) \subseteq N(\int_0^te^{\lambda s}ds-S(t)) \subseteq R^{\infty}(\int_0^te^{\lambda s}ds-S(t)) \subseteq R^{\infty}(\lambda -A)$$ Therefore $\lambda- A$ is a Saphar operator.\\
Let $M$ a   finite dimensional subspace of $X$.
We have, $$N(\lambda -A) \subseteq N(\int_0^te^{\lambda s}ds-S(t)) \subseteq R^{\infty}(\int_0^te^{\lambda s}ds-S(t))+M \subseteq R^{\infty}(\lambda -A)+M$$ Hence  $\int_0^te^{\lambda s}ds-S(t)$ is  essentially Saphar implies that $\lambda-A$ is.
\end{proof}

\section{Spectral Inclusion For Quasi-Fredholm Spectrum }
Recall from \cite{lab} some definitions:
\begin{definition}
Let $T$  be a closed linear operator on $X$ and let
$$\Delta(T)=\{n\in\mathbb{N}, \forall m\geq n, R(T^n)\cap N(T)=R(T^m)\cap N(T)\}$$
The degree of stable  iteration $dis(T)$ of $T$ is defined as $dis(T)=inf \Delta(T)$ which $dis(T)=\infty$ if $\Delta(T)=\emptyset$.
\end{definition}

\begin{definition}
Let $T$  be a closed linear operator on $X$. $T$ is called a quasi-Fredholm operator of degree $d$ if there exists an integer  $d\in \mathbb{N}$ such that :
\begin{enumerate}
  \item $dis(T)=d$;
  \item $R(T^n)$ is closed in $X$ for all $n\geq d$;
  \item $R(T)+N(T^n)$ is closed in $X$ for all $n\geq d$.
\end{enumerate}
\end{definition}

The quasi-Fredholm spectrum is defined by :
$$\sigma_{qF}(T)=\{\lambda\in\mathbb{C}: T-\lambda I\mbox{  is not  a quasi-Fredholm }\}$$
\begin{proposition}\label{b}
Let $(S(t))_{t\geq 0}$ be a non-degenerate integrated semigroup with generator $A$. Then
\begin{center}
$dis(\lambda - A)\leq dis(\int_0^te^{\lambda s}ds-S(t))$
\end{center}
\end{proposition}
\begin{proof}
If $dis(\int_0^te^{\lambda s}ds-S(t))= + \infty$ obvious.\\\\
If $dis(\int_0^te^{\lambda s}ds-S(t))= d\in\mathbb{N^{*}}$ then for all $n\geq d$, we have :\\
$R((\int_0^te^{\lambda s}ds-S(t))^n)\cap N((\int_0^te^{\lambda s}ds-S(t)))=R((\int_0^te^{\lambda s}ds-S(t))^d)\cap N(\int_0^te^{\lambda s}ds-S(t))$
Then for all $n\geq d$, we have :\\
$$R((\lambda- A)^n)\cap N(\lambda- A)=R((\lambda- A)^d)\cap N(\lambda- A)$$
 Indeed, let $y\in R((\lambda- A)^d)\cap N(\lambda- A)$ there exists $x\in X$ such that $y=(\lambda- A)^{d}x$ and by lemma \ref{l1}, we can show that  there exists two operators  $F_{d}(t)$ and $G_{d}(t)$ such that
$$(\lambda- A)^{d} F_{d}(t)+ D^{d}_{\lambda}(t)G_{d}(t)=I$$
implies that $(\lambda-A)^{d}x=(\lambda- A)^{2d} F_{d}(t)x+ (\int_0^te^{\lambda s}ds-S(t))^d G_{d}(t)x$, since $y\in N(\lambda-A)$,
$y=(\lambda-A)^{d}x=(\int_0^te^{\lambda s}ds-S(t))^n z= (\lambda-A)^{n} D_{\lambda}^{n}(t)z \in R((\lambda- A)^n)\cap N(\lambda- A)$
therefore, $dis(\lambda - A)\leq d$.\\\\
If $d=0$, for all $n\geq d$, we have :\\
$R((\int_0^te^{\lambda s}ds-S(t))^n)\cap N((\int_0^te^{\lambda s}ds-S(t)))= N(\int_0^te^{\lambda s}ds-S(t))$, then\\
 $\forall n\in\mathbb{N}$ $N((\int_0^te^{\lambda s}ds-S(t)))\subset R((\int_0^te^{\lambda s}ds-S(t))^n)$, then
 $N(\lambda-A)\subset N((\int_0^te^{\lambda s}ds-S(t)))\subset R((\int_0^te^{\lambda s}ds-S(t))^n)\subset R((\lambda-A)^n)$ hence
 $N(\lambda-A)\cap R((\lambda-A)^n)=N(\lambda-A)\cap R((\lambda-A)^0)$, therefore $dis(\lambda - A)=0$
 \end{proof}

 \begin{proposition}
Let $(S(t))_{t\geq 0}$ be  a non-degenerate integrated semigroup on $X$ with  generator $A$.

 If $R((\int_0^te^{\lambda s}ds-S(t))^n)$  is closed for all $n\geq d$, then $R( (\lambda-A)^n)$ is closed for all $n\geq d$

\end{proposition}
\begin{proof}
 Let $y_p=(\lambda-A)^nx_p\rightarrow y$, as $p\rightarrow \infty$, we show that $y\in R((\lambda-A)^n)$.
 According to lemma \ref{l1} there exists $F_{n}(t)$ and $G_{n}(t)$ two bounded linear operators such that
 $$   (1):   (\lambda- A)^{n} F_{n}(t)+ D^{n}_{\lambda}(t)G_{n}(t)=I$$
 $D^{n}_{\lambda}(t)y_p= D^{n}_{\lambda}(t)(\lambda-A)^n x_p=(\int_0^te^{\lambda s}ds-S(t))^{n}x_p\in R((\int_0^te^{\lambda s}ds-S(t))^n)$
 since $D^{n}_{\lambda}(t)y_p\rightarrow D^{n}_{\lambda}(t)y$, then there exists $z\in X$ such that
 $D^{n}_{\lambda}(t)y=(\int_0^te^{\lambda s}ds-S(t))^{n}z$.
 By $(1)$ :\\$(\lambda-A)^{n}x_p=(\lambda-A)^{2n}F_{n}(t)x_p+ (\int_0^te^{\lambda s}ds-S(t))^{n} G_{n}(t)x_p$
 as $n\rightarrow \infty$, we have :
 \begin{eqnarray*}
            y &=& (\lambda-A)^{n}F_{n}(t)y+ (\int_0^te^{\lambda s}ds-S(t))^{n}G_{n}(t)z \\
              &=& (\lambda-A)^{n}[F_{n}(t)y+ D^{n}_{\lambda}(t)G_{n}(t)z]\in R((\lambda-A)^n)
\end{eqnarray*}
\end{proof}

 \begin{proposition}
Let $(S(t))_{t\geq 0}$ be  a non-degenerate integrated semigroup on $X$ with  generator $A$ and $d\in\mathbb{N}$.

 If $R(\int_0^te^{\lambda s}ds-S(t))+ N((\int_0^te^{\lambda s}ds-S(t))^d)$  is closed in $X$, then $R(\lambda-A )+N((\lambda-A )^d)$ is closed

\end{proposition}
\begin{proof}
Suppose that $R(\int_0^te^{\lambda s}ds-S(t))+ N((\int_0^te^{\lambda s}ds-S(t))^d)$  is closed in $X$.\\
Let $y_n=(\lambda-A)x_n+ z_n\rightarrow y$, as $n\rightarrow \infty$, $z_n\in N((\lambda-A)^d)$\\
$D^{d}_{\lambda}(t)y_n=D^{d}_{\lambda}(t)(\lambda-A)x_n+D^{d}_{\lambda}(t)z_n
\in R(\int_0^te^{\lambda s}ds-S(t))+ N((\int_0^te^{\lambda s}ds-S(t))^d)$, then $D^{d}_{\lambda}(t)y\in R(\int_0^te^{\lambda s}ds-S(t))+ N((\int_0^te^{\lambda s}ds-S(t))^d)$\\
$D^{d}_{\lambda}(t)y=((\int_0^te^{\lambda s}ds-S(t))x+z$ with $z\in N((\int_0^te^{\lambda s}ds-S(t))^d)$\\

$D^{2d}_{\lambda}(t)y=D^{d}_{\lambda}(t)(\int_0^te^{\lambda s}ds-S(t))x+ D^{d}_{\lambda}(t)z$ and  $D^{d}_{\lambda}(t)z\in N((\lambda-A)^d)$.\\ Then
\begin{eqnarray*}
  y &=& (\lambda-A)^{2d}F_{d}(t)y+D_{\lambda}^{2d}(t)G_{d}(t)y \\
    &=& (\lambda-A)^{2d}F_{d}(t)y+G_{d}(t)[D_{\lambda}^{d}(t)(\int_0^te^{\lambda s}ds-S(t))x+ D^{d}_{\lambda}(t)z] \\
    &=& (\lambda-A)[(\lambda-A)^{2d-1}F_{d}(t)y+G_{d}(t)D_{\lambda}^{d-1}(t)]+G_{d}(t)D_{\lambda}^{d}(t)z\in  R(\lambda-A)+N((\lambda-A)^d)
\end{eqnarray*}

\end{proof}
\begin{corollary}\label{a}
Let $(S(t))_{t\geq 0}$ be a non-degenerate integrated semigroup on $X$ with  generator $A$.

 If $R(\int_0^te^{\lambda s}ds-S(t))$ is closed in $X$, then $R(\lambda-A )$ is closed
\end{corollary}
\begin{theorem}

Let $(S(t))_{t\geq 0}$  be a non-degenerate integrated semigroup on $X$ with  generator $A$. Then for all $t> 0$:
\begin{center}
$\int_0^t e^{t\sigma_{qF}(A)}ds \subseteq \sigma_{qF}(S(t))$
\end{center}
\end{theorem}
\begin{proof}
Direct consequence of the three last proposition.
\end{proof}
\begin{remark}
If $dis(T)=0$,  $T$ to be a Kato operator and by using the  corollary \ref{a}  and proposition \ref{b} we have for all $t> 0$:
$$\int_0^t e^{t\sigma_{K}(A)}ds \subseteq \sigma_{K}(S(t))$$
\end{remark}

\end{document}